\documentclass[a4paper]{article}

\pdfoutput=1

\usepackage[hyphens]{url}

\usepackage{colonequals}
\usepackage{caption}
\usepackage[position=b]{subcaption}

\usepackage{amsmath, amsthm, amssymb}
\theoremstyle{plain}
\numberwithin{equation}{section} 
\newtheorem{theorem}[equation]{Theorem}
\newtheorem{proposition}[equation]{Proposition}
\newtheorem{lemma}[equation]{Lemma}
\newtheorem{conjecture}[equation]{Conjecture}

\theoremstyle{definition}
\newtheorem{definition}[equation]{Definition}
\newtheorem{fact}[equation]{Fact}
\newtheorem*{keywords}{keywords}

\newtheoremstyle{dotless}{}{}{}{}{\bfseries}{}{ }{}
\theoremstyle{dotless}
\newtheorem*{ccode}{Mathematics Subject Classification 2010:}

\usepackage{tikz}
\usetikzlibrary{calc}
\usetikzlibrary{shapes.geometric}

\usepackage{listings}
\usepackage{pgfplots}

\DeclareMathOperator{\Mod}{Mod^+}
\DeclareMathOperator{\GL}{GL}
\DeclareMathOperator{\poly}{poly}
\newcommand{\PML}{\mathcal{PML}}
\newcommand{\calL}{\mathcal{L}}
\newcommand{\FF}{\mathbb{F}}
\newcommand{\ZZ}{\mathbb{Z}}
\newcommand{\RP}{\mathbb{RP}}
\newcommand{\QQbar}{\overline{\mathbb{Q}}}
\newcommand{\defeq}{\colonequals}

\newcommand{\superscript}[1]{\ensuremath{^{\textrm{#1}}}}
\newcommand{\nth}{\superscript{th}}
\newcommand{\inlineand}{\quad \textrm{and} \quad}
\newcommand{\inlineQED}{\pushQED{\qed} \qedhere \popQED}

\newcommand\blfootnote[1]{%
  \begingroup
  \renewcommand\thefootnote{}\footnote{#1}%
  \addtocounter{footnote}{-1}%
  \endgroup
}

\DeclareMathOperator{\height}{hgt}
\DeclareMathOperator{\spectralratio}{\omega}

\pgfplotsset{
    box plot/.style={
        /pgfplots/.cd,
        black,
        only marks,
        mark=-,
        mark size=\pgfkeysvalueof{/pgfplots/box plot width},
        /pgfplots/error bars/y dir=plus,
        /pgfplots/error bars/y explicit,
        /pgfplots/table/x index=\pgfkeysvalueof{/pgfplots/box plot x index},
    },
    box plot box/.style={
        /pgfplots/error bars/draw error bar/.code 2 args={%
            \draw  ##1 -- ++(\pgfkeysvalueof{/pgfplots/box plot width},0pt) |- ##2 -- ++(-\pgfkeysvalueof{/pgfplots/box plot width},0pt) |- ##1 -- cycle;
        },
        /pgfplots/table/.cd,
        y index=\pgfkeysvalueof{/pgfplots/box plot box top index},
        y error expr={
            \thisrowno{\pgfkeysvalueof{/pgfplots/box plot box bottom index}}
            - \thisrowno{\pgfkeysvalueof{/pgfplots/box plot box top index}}
        },
        /pgfplots/box plot
    },
    box plot top whisker/.style={
        /pgfplots/error bars/draw error bar/.code 2 args={%
            \pgfkeysgetvalue{/pgfplots/error bars/error mark}%
            {\pgfplotserrorbarsmark}%
            \pgfkeysgetvalue{/pgfplots/error bars/error mark options}%
            {\pgfplotserrorbarsmarkopts}%
            \path ##1 -- ##2;
        },
        /pgfplots/table/.cd,
        y index=\pgfkeysvalueof{/pgfplots/box plot whisker top index},
        y error expr={
            \thisrowno{\pgfkeysvalueof{/pgfplots/box plot box top index}}
            - \thisrowno{\pgfkeysvalueof{/pgfplots/box plot whisker top index}}
        },
        /pgfplots/box plot
    },
    box plot bottom whisker/.style={
        /pgfplots/error bars/draw error bar/.code 2 args={%
            \pgfkeysgetvalue{/pgfplots/error bars/error mark}%
            {\pgfplotserrorbarsmark}%
            \pgfkeysgetvalue{/pgfplots/error bars/error mark options}%
            {\pgfplotserrorbarsmarkopts}%
            \path ##1 -- ##2;
        },
        /pgfplots/table/.cd,
        y index=\pgfkeysvalueof{/pgfplots/box plot whisker bottom index},
        y error expr={
            \thisrowno{\pgfkeysvalueof{/pgfplots/box plot box bottom index}}
            - \thisrowno{\pgfkeysvalueof{/pgfplots/box plot whisker bottom index}}
        },
        /pgfplots/box plot
    },
    box plot median/.style={
        /pgfplots/box plot,
        /pgfplots/table/y index=\pgfkeysvalueof{/pgfplots/box plot median index}
    },
    box plot width/.initial=1em,
    box plot x index/.initial=0,
    box plot median index/.initial=1,
    box plot box top index/.initial=2,
    box plot box bottom index/.initial=3,
    box plot whisker top index/.initial=4,
    box plot whisker bottom index/.initial=5,
}

\newcommand{\boxplot}[2][]{
    \addplot [box plot median,#1] table {#2};
    \addplot [forget plot, box plot box,#1] table {#2};
    \addplot [forget plot, box plot top whisker,#1] table {#2};
    \addplot [forget plot, box plot bottom whisker,#1] table {#2};
}

\usepackage[pdfborderstyle={},pdfborder={0 0 0},pagebackref,pdftex]{hyperref}

\renewcommand*{\backref}[1]{}
\renewcommand*{\backrefalt}[4]{
  \ifcase #1 %
   [No citations.]%
  \or
   [#2]%
  \else
   [#2]%
  \fi
}

\author{Mark C. Bell \and Saul Schleimer}

\title{Slow north-south dynamics on $\PML$}

\begin{document}

\maketitle

\begin{abstract}
We consider the action of a pseudo-Anosov mapping class on $\PML(S)$. This action has north-south dynamics and so, under iteration, laminations converge exponentially to the stable lamination.

We study the rate of this convergence and give examples of families of pseudo-Anosov mapping classes where the rate goes to one, decaying exponentially with the word length. Furthermore we prove that this behaviour is the worst possible. \blfootnote{This work is in the public domain.} \blfootnote{Department of Mathematics, University of Illinois: \texttt{mcbell@illinois.edu}}
\blfootnote{Mathematics Institute, University of Warwick: \texttt{s.schleimer@warwick.ac.uk}}
\end{abstract}


\ccode{37E30, 57M99}

\section{Introduction}

A pseudo-Anosov mapping class $h \in \Mod(S)$ acts on $\PML(S)$ with north-south dynamics. Therefore its action has a pair of fixed points $\calL^{\pm}(h) \in \PML(S)$ and under iteration laminations (other than $\calL^-(h)$) converge to $\calL^+(h)$. A pants decomposition, collection of train tracks or ideal triangulation gives a coordinate system on $\PML(S)$ \cite[Expos\'{e}~6]{FLPtranslated}. In any such system the convergence to $\calL^+(h)$ is exponential.

Thurston suggested that under iteration laminations always converge to $\calL^+(h)$ ``rather quickly''~\cite[Page~427]{Thurston}. If this were true for all pseudo-Anosov mapping classes then iteration would give an efficient algorithm to find $\calL^+(h)$. However it is false:

\begin{theorem}
\label{thrm:main}
Suppose that $3g - 3 + p \geq 4$ and fix a finite generating set for $\Mod(S_{g,p})$. There is an infinite family of pseudo-Anosov mapping classes where the rate of convergence goes to one, and decays exponentially with respect to word length.
\end{theorem}

As usual, we use $S_{g,p}$ to denote the surface of genus $g$ with $p$ punctures.

In Section~\ref{sec:upper_bounds} we show how to construct such a family on $S_{3,0}$ and in Section~\ref{sec:other_lower} we show how to generalise this construction to other surfaces. Furthermore, in Proposition~\ref{prop:spectral_ratio_lower} we show that this type of convergence is the worst possible. Finally in Section~\ref{sec:flipper} we describe how these examples can be rigorously verified using \texttt{flipper} \cite{flipper}.

In order to bound the rate of convergence, we use the following definition.

\begin{definition}
Suppose that $f \in \ZZ[x]$ is a polynomial with roots $\lambda_1, \ldots, \lambda_m$, ordered such that $|\lambda_1| \geq |\lambda_2| \geq \cdots \geq |\lambda_m|$. The \emph{spectral ratio} of $f$ is $\spectralratio(f) \defeq |\lambda_1 / \lambda_2|$.
\end{definition}

This is motivated by the equivalent problem of $\GL(N, \ZZ)$ acting on $\RP^{N-1}$. Under iteration of a matrix $M \in \GL(N, \ZZ)$, generic vectors in $\RP^{N-1}$ converge exponentially to the dominant eigenvector of $M$. The rate of this convergence is bounded above by $\spectralratio(M) \defeq \spectralratio(\chi_M)$, the spectral ratio of the characteristic polynomial of $M$ \cite[Section~4.1]{Watkins}.

\begin{definition}
Suppose that $h \in \Mod(S)$ is a pseudo-Anosov mapping class. Let $\mu_{\lambda(h)} \in \ZZ[x]$ denote the minimal polynomial of its \emph{dilatation} $\lambda(h)$. The \emph{spectral ratio} of $h$ is $\spectralratio(h) \defeq \spectralratio(\mu_{\lambda(h)})$.
\end{definition}

Choose one of the above coordinate systems on $\PML(S)$ and a pseudo-Anosov mapping class $h$.
On a suitable neighbourhood of $\calL^+(h)$, the action of $h$ is given by an integer matrix $M$.
Hence, for a generic lamination $\calL$ the rate of convergence of $h^n(\calL)$ to $\calL^+(h)$ is bounded above by $\spectralratio(M)$.
However, the dominant eigenvalue of $M$ is $\lambda(h)$ and so $\spectralratio(M) \leq \spectralratio(h)$.
Thus we achieve Theorem~\ref{thrm:main} by producing mapping classes with spectral ratio exponentially close to one.

Finally, we conjecture that the surface complexity condition in Theorem~\ref{thrm:main} is not only sufficient but also necessary. If so then this problem is subtly different from the equivalent problem for matrices. In $\GL(N, \ZZ)$ where exponentially slow convergence occurs even when $N = 3$. One such family is given by the matrices
\[ \left( \begin{array}{ccc}
 0 & 0 & 1 \\
 1 & 0 & 2^k \\
 0 & 1 & 0
\end{array} \right) \]
These have word length only $O(k)$, due to distorted subgroups~\cite[Theorem~4.1]{Riley}.

\section{An upper bound by example}
\label{sec:upper_bounds}


We start by constructing an explicit family of pseudo-Anosov mapping classes on $S_{3,0}$ whose spectral ratio goes to one exponentially with the word length. To do this we use:
\begin{itemize}
\item $\varphi \defeq \frac{1 + \sqrt{5}}{2}$ to denote the \emph{golden ratio},
\item $F_n$ to denote the $n$\nth{} \emph{Fibonacci number}, and
\item $x \approx_t y$ to denote that $|x - y| \leq t$.
\end{itemize}

\begin{figure}[ht]
\centering
\newdimen\R
\R=0.8cm
\begin{tikzpicture}[scale=1.85,thick]
  \draw (30:\R) \foreach \x in {30,90,...,330} {-- (\x:\R)}-- cycle;
  \draw [xshift=2*0.86602540378*\R] (30:\R) \foreach \x in {30,90,...,330} {-- (\x:\R)}-- cycle;
  \draw [xshift=-2*0.86602540378*\R] (30:\R) \foreach \x in {30,90,...,330} {-- (\x:\R)}-- cycle;

  \foreach \x/\y in {60/5,120/4,240/5,300/4} {\node at (\x:\R) {$\y$};};
  \begin{scope}[xshift=2*0.86602540378*\R]
    \foreach \x/\y in {0/1,60/7,120/6,240/7,300/6} {\node at (\x:\R) {$\y$};};
  \end{scope}
  \begin{scope}[xshift=-2*0.86602540378*\R]
    \foreach \x/\y in {60/3,120/2,180/1,240/3,300/2} {\node at (\x:\R) {$\y$};};
  \end{scope}
  
  \foreach \i in {-1,0,1} {
    \begin{scope}[xshift=\i*2*0.86602540378*\R]
      \foreach \x in {30,90,...,330} {\node (A\i N\x) at (\x:\R) {};};
    \end{scope} 
  };
  
  \begin{scope}[xshift=2*0.86602540378*\R]
    \foreach \x/\y in {0/1,60/7,120/6,240/7,300/6} {\node at (\x:\R) {$\y$};};
  \end{scope}
  \begin{scope}[xshift=-2*0.86602540378*\R]
    \foreach \x/\y in {60/3,120/2,180/1,240/3,300/2} {\node at (\x:\R) {$\y$};};
  \end{scope}
  
  \node (astart) at ($(A0N90)!0.5!(A0N150)$) {}; \node (aend) at ($(A0N270)!0.5!(A0N330)$) {};
  \draw [red] (astart.center) -- (aend.center);
  \node [above, right] at ($(astart)!0.75!(aend)$) {$a$};
  
  \node (bstart) at ($(A0N30)!0.5!(A0N90)$) {}; \node (bend) at ($(A0N210)!0.5!(A0N270)$) {};
  \draw [red] (bstart.center) -- (bend.center);
  \node [above, right] at ($(bstart)!0.25!(bend)$) {$b$};
  
  \path [draw, red] ($(A-1N30)!0.25!(A-1N90)$) to [out=-120,in=180]
      ($(A-1N30)!0.33!(A-1N330)$) to [out=0, in=-120]
      ($(A0N30)!0.75!(A0N90)$);
  \path [draw, red] ($(A-1N210)!0.75!(A-1N270)$) to [out=60,in=180] node [above, black] {$c$}
      ($(A-1N330)!0.33!(A-1N30)$) to [out=0, in=60]
      ($(A0N270)!0.75!(A0N210)$);
  
  \draw [->] ($(A0N330) + (0,-0.75)$) -- node [above] {$\rho$} ($(A0N210) + (0,-0.75)$);
\end{tikzpicture}
\caption{Curves on the surface of genus $3$.}
\label{fig:surface}
\end{figure}
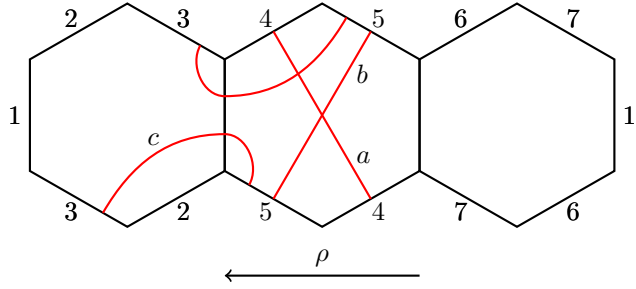

Let $S$ be the surface of genus $3$ as shown in Figure~\ref{fig:surface} in which sides with the same label are identified. Fix $k \geq 7$ such that $k \equiv 2 \pmod{8}$. Consider the mapping class
\[ h \defeq \rho \circ T_c \circ (T_a^{-1} \circ T_b)^k \]
where $T_x$ denotes a right Dehn twist about $x$ and $\rho$ is the order three mapping class which cycles these hexagons to the left.

\begin{theorem}
\label{thrm:spectral_ratio}
The mapping class $h$ is pseudo-Anosov and $\spectralratio(h) \leq 1 + 14 \varphi^{-k}$.
\end{theorem}

\begin{proof}
First note that as in \cite[Page~448]{PennerLeastDilatations} it follows immediately from \cite[Theorem~3.1]{PennerConstruction} that $h^3$ is pseudo-Anosov and so $h$ is too.

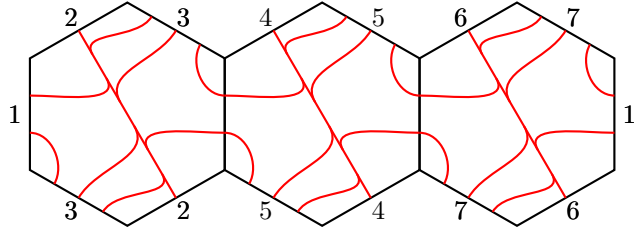
\begin{figure}[ht]
\centering
\newdimen\R
\R=0.8cm
\begin{tikzpicture}[scale=1.85,thick]

  \draw (30:\R) \foreach \x in {30,90,...,330} {-- (\x:\R)} -- cycle;
  \draw [xshift=2*0.86602540378*\R] (30:\R) \foreach \x in {30,90,...,330} {-- (\x:\R)} -- cycle;
  \draw [xshift=-2*0.86602540378*\R] (30:\R) \foreach \x in {30,90,...,330} {-- (\x:\R)} -- cycle;

  \foreach \x/\y in {60/5,120/4,240/5,300/4} {\node at (\x:\R) {$\y$};};
  \begin{scope}[xshift=2*0.86602540378*\R]
    \foreach \x/\y in {0/1,60/7,120/6,240/7,300/6} {\node at (\x:\R) {$\y$};};
  \end{scope}
  \begin{scope}[xshift=-2*0.86602540378*\R]
    \foreach \x/\y in {60/3,120/2,180/1,240/3,300/2} {\node at (\x:\R) {$\y$};};
  \end{scope}
  
  \foreach \i in {-1,0,1} {
    \begin{scope}[xshift=\i*2*0.86602540378*\R]
      \foreach \x in {30,90,...,330} {\node (A\i N\x) at (\x:\R) {};};
    \end{scope} 
  };
  
  \begin{scope}[xshift=2*0.86602540378*\R]
    \foreach \x/\y in {0/1,60/7,120/6,240/7,300/6} {\node at (\x:\R) {$\y$};};
  \end{scope}
  \begin{scope}[xshift=-2*0.86602540378*\R]
    \foreach \x/\y in {60/3,120/2,180/1,240/3,300/2} {\node at (\x:\R) {$\y$};};
  \end{scope}
  
  \foreach \i in {-1,0,1} {
    \node (astart) at ($(A\i N90)!0.5!(A\i N150)$) {}; \node (aend) at ($(A\i N270)!0.5!(A\i N330)$) {};
    \draw [red] (astart.center) -- (aend.center);
    
    \draw [red] ($(astart.center)!1/7!(aend.center)$) to [out=120, in=240] ($(A\i N90)!1/4!(A\i N30)$);
    \draw [red] ($(astart.center)!2/7!(aend.center)$) to [out=300, in=0] ($(A\i N150)!1/3!(A\i N210)$);
    \draw [red] ($(astart.center)!3/7!(aend.center)$) to [out=120, in=240] ($(A\i N90)!2/4!(A\i N30)$);
    \draw [red] ($(astart.center)!4/7!(aend.center)$) to [out=300, in=60] ($(A\i N210)!2/4!(A\i N270)$);
    \draw [red] ($(astart.center)!5/7!(aend.center)$) to [out=120, in=180] ($(A\i N330)!1/3!(A\i N30)$);
    \draw [red] ($(astart.center)!6/7!(aend.center)$) to [out=300, in=60] ($(A\i N210)!3/4!(A\i N270)$);
    
    \draw [red] ($(A\i N330)!2/3!(A\i N30)$) to [out=180, in=240] ($(A\i N90)!3/4!(A\i N30)$);
    \draw [red] ($(A\i N150)!2/3!(A\i N210)$) to [out=0, in=60] ($(A\i N210)!1/4!(A\i N270)$);
  }; 
\end{tikzpicture}
\caption{The invariant train track $\tau$ of $h$ \cite[Figure~3b]{PennerLeastDilatations}.}
\label{fig:train_track}
\end{figure}

Now the train track $\tau$, shown in Figure~\ref{fig:train_track}, is invariant under $h$. Direct calculation show that the hitting matrix of $\tau$ under $h$ with respect to the basis $a, b, c, \rho(a), \rho(b), \rho(c), \rho^2(a), \rho^2(b), \rho^2(c)$ is:
\[ M \defeq \left( \begin{array}{ccc|ccc|ccc}
0 & 0 & 0 & 0 & 0 & 0 & 1 & 0 & 0 \\
0 & 0 & 0 & 0 & 0 & 0 & 0 & 1 & 0 \\
0 & 0 & 0 & 0 & 0 & 0 & 0 & 0 & 1 \\ \hline
F_{2k+1} & F_{2k} & F_{2k} & 0 & 0 & 0 & 0 & 0 & F_{2k} \\
F_{2k} & F_{2k-1} & F_{2k-1} - 1 & 0 & 0 & 0 & 0 & 0 & F_{2k-1} - 1 \\
F_{2k+1} & F_{2k} & F_{2k} +1 & 1 & 0 & 0 & 0 & 0 & F_{2k} \\ \hline
0 & 0 & 0 & 1 & 0 & 0 & 0 & 0 & 0 \\
0 & 0 & 0 & 0 & 1 & 0 & 0 & 0 & 0 \\
0 & 0 & 0 & 0 & 0 & 1 & 0 & 0 & 0
\end{array} \right) \]
and that the characteristic polynomial of $M$ is
\[ (x^3 - 1) (x^6 - F_{2k} x^4 - F_{2k+3} x^3 - F_{2k} x^2 + 1). \]
Using the fact that $k \equiv 2 \pmod{8}$, reducing the right-hand factor of this modulo $7$ we obtain $x^6 + 4x^4 + x^3 + 4x^2 + 1 \in \FF_{7}[x]$. This is irreducible in $\FF_{7}[x]$ and so the minimal polynomial of $\lambda(h)$ is
\[ \mu_{\lambda(h)}(x) = x^6 - F_{2k} x^4 - F_{2k+3} x^3 - F_{2k} x^2 + 1. \]

To find the roots of $\mu_{\lambda(h)}$, we divide by $x^3$ and substitute $y \defeq x+ x^{-1}$ to obtain:
\[ y^3 - (F_{2k} + 3) y - F_{2k+3}. \]
Let $y_{-1}$, $y_0$ and $y_{1}$ be the three roots of this cubic. These are all real numbers as $\Delta = 4 (F_{2k} + 3)^3 - 27 F_{2k+3}^2> 0$ and using the cubic Vi\`{e}te formula are given by:
\[ y_j \defeq \frac{2}{\sqrt{3}} \sqrt{F_{2k}+3} \cos \left(\frac{1}{3} \arccos \left( \frac{3 \sqrt{3} F_{2k+3}}{(F_{2k} + 3)\sqrt{F_{2k} + 3}} \right) - (j-1) \frac{2\pi}{3} \right) \]
By using the Taylor series for cosine and arccosine together with the fact that $\sqrt{F_{2k} + 3} \approx_1 \sqrt[4]{5} F_k$, we deduce that
\[ y_{-1} \approx_2 -\sqrt[4]{5} F_k, \; y_0 \approx_2 0 \; \textrm{and} \; y_1 \approx_2 \sqrt[4]{5} F_k. \]

Since $y = x+ x^{-1}$, the six roots of $\mu_{\lambda(h)}$ are given by
\[ x_i^{\pm} \defeq \frac{y_i \pm \sqrt{y_i^2 - 4}}{2} \]
Thus for each $i$ either $x_i^+ \approx_1 y_i$ and $x_i^- \approx_1 0$, or vice versa. In particular
\[ x_{-1}^- \approx_3 -\sqrt[4]{5} F_k \inlineand x_1^+ \approx_3 \sqrt[4]{5} F_k. \]
The other four roots of $\mu_{\lambda(h)}$ lie in $B(0, 3)$, the disk about $0$ of radius $3$. So we deduce that the spectral ratio of $h$ is at most the ratio of $|x_{-1}^-|$ and $|x_1^+|$. As these both lie in $X \defeq B(\sqrt[4]{5} F_k, 3)$, we therefore have that
\[ \spectralratio(h) \leq \frac{\max(X)}{\min(X)} \leq \frac{\sqrt[4]{5} F_k + 3}{\sqrt[4]{5} F_k - 3} \leq 1 + \frac{6}{F_{k}} \leq 1 + 14 \varphi^{-k}. \qedhere \]
\end{proof}

\section{Other surfaces}
\label{sec:other_lower}

We also consider the possible spectral ratios of pseudo-Anosov mapping classes on other surfaces. We summarise the results of this section in Table~\ref{tab:tikz_ratios} where:
\begin{itemize}
\item N denotes that there are no pseudo-Anosov mapping classes,
\item B denotes that the spectral ratios of pseudo-Anosov mapping classes are bounded away from one,
\item P$_\leq$ denotes that $\spectralratio(h) \leq 1 + \frac{1}{\poly(|h|)}$ for some infinite family of pseudo-Anosov mapping classes, and
\item E denotes that $\spectralratio(h) \leq 1 + \frac{1}{\exp(|h|)}$ for some infinite family of pseudo-Anosov mapping classes.
\end{itemize}

\begin{table}[ht]
\centering
\begin{tikzpicture}[scale=0.85,thick]

\node at (4.5,0.75) {punctures};
\foreach \i in {0,1,...,7} {\node at (\i+1,0) {$\i$};};
\foreach \i in {0,1,...,3} {\node at (0,-\i-1) {$\i$};};

\node[rotate=90] at (-0.75,-3) {genus};

\draw (0.5,0.5) -- (0.5,-4.5);
\draw (-0.5,-0.5) -- (8.5,-0.5);

\path [draw] (0.5,-1.5) -- (4.5,-1.5) -- (4.5,-0.5);
\node at (2.5,-1) {N};

\path [draw] (0.5,-2.5) -- (2.5,-2.5) -- (2.5,-1.5);
\path [draw] (4.5,-1.5) -- (5.5,-1.5) -- (5.5,-0.5);
\node at (5,-1) {B}; \node at (1.5,-2) {B};

\path [draw] (0.5,-3.5) -- (1.5,-3.5) -- (1.5,-2.5);
\path [draw] (2.5,-2.5) -- (4.5,-2.5) -- (4.5,-1.5);
\path [draw] (5.5,-1.5) -- (7.5,-1.5) -- (7.5,-0.5);
\node at (6.5,-1) {P$_\leq$}; \node at (3.5,-2) {P$_\leq$}; \node at (1,-3) {P$_\leq$};

\node at (5,-3) {E};

\end{tikzpicture}
\caption{Spectral ratios in other surfaces.}
\label{tab:tikz_ratios}
\end{table}

\begin{conjecture}
In all of the P$_\leq$ cases, there is no family of pseudo-Anosov mapping classes whose spectral ratios converge to one exponentially. That is, none of the P$_\leq$ cases are actually E cases.
\end{conjecture}

Note that in all of our examples of slow convergence there is a pair of identical disjoint subsurfaces supporting much of the dynamics of the mapping class. Furthermore, some of the topology of $S$ lies outside of these subsurfaces. If these are necessary conditions then, since Dehn twists subgroups are undistorted in $\Mod(S)$ \cite[Theorem~1.1]{FarbLubotzkyMinsky}, the conjecture should follow.

\subsection{Exponential convergence}
\label{sub:exp}

We start by considering the cases where $3g - 3 + p = 4$. Here the same argument as in Theorem~\ref{thrm:spectral_ratio} shows that a similar exponential spectral ratio bound also holds for:
\begin{itemize}
\item $S_{0,7}$ via the (spherical) braid $\sigma_4^{-1} (\sigma_5 \sigma_6^{-1})^k (\sigma_1 \sigma_2^{-1})^k \sigma_3$,
\item $S_{1,4}$ via $T_d^{-1} \circ (T_e \circ T_f^{-1})^k \circ (T_a \circ T_{b}^{-1})^k \circ T_c$, and
\item $S_{2,1}$ via $T_d^{-1} \circ (T_e \circ T_f^{-1})^k \circ (T_a \circ T_b^{-1})^k \circ T_c$.
\end{itemize}
The curves used for $S_{1,4}$ and $S_{2,1}$ are shown in Figure~\ref{fig:S_1_4} and Figure~\ref{fig:S_2_1} respectively.

\begin{figure}[ht]
\centering
\subcaptionbox{Curves on $S_{1,4}$. \label{fig:S_1_4}}{\begin{tikzpicture}[scale=0.75, thick]

\draw [dotted, red] (0, 1.35) to [out=-70, in=180+70] (1.8,1.35);
\draw [dotted, red] (3.95,0) to [out=180,in=-70] (0.2 + 3, 1.35);
\draw [dotted, red] (2.2,0) to [out=0,in=180] (2.9,1.25);

\draw [rounded corners=20] (0,0) rectangle (7, 3);

\draw (-1.5 + 3,1.5) to [out=-30, in=210] (0.5 + 3, 1.5);
\draw (-1.2 + 3,1.35) to [out=30, in=150] (0.2 + 3, 1.35);

\draw [red] (2.9, 1.25) to [out=0,in=270] (3.7, 1.5) to [out=90,in=0] (2.5,2) to [out=180,in=90] (1.3,1.5) to [out=270,in=180] (2.5,1.0) to [out=0,in=270] (3.9,1.5) to [out=90,in=0] (2.5,2.2) to [out=180,in=90] (1.1,1.5) to [out=270,in=180] (2.2,0);

\draw [red] (0, 1.35) to [out=70, in=180-70] (1.8,1.35);
\draw [red] (0.2 + 3, 1.35) to [out=70,in=90] (4.5 + 0.2,1.5) to [out=270,in=0] (3.95,0);

\node at (0.9, 2.1) {$a$};
\node at (2.5, 2.5) {$b$};
\node at (4.125,2.3) {$c$};
\node at (4.875,2.3) {$d$};
\node at (5.625,2.3) {$e$};
\node at (6.375,2.3) {$f$};

\node [draw,shape=circle,fill=black,scale=0.4] at (4.5, 1.5) {};
\node [draw,shape=circle,fill=black,scale=0.4] at (5.25, 1.5) {};
\node [draw,shape=circle,fill=black,scale=0.4] at (6, 1.5) {};
\node [draw,shape=circle,fill=black,scale=0.4] at (6.75, 1.5) {};

\draw [red] (4.875,1.5) circle (0.55 and 0.5);
\draw [red] (5.625,1.5) circle (0.55 and 0.5);
\draw [red] (6.375,1.5) circle (0.55 and 0.5);

\end{tikzpicture}}
\subcaptionbox{Curves on $S_{2,1}$. \label{fig:S_2_1}}{\begin{tikzpicture}[scale=0.75, thick]

\draw [dotted, red] (0, 1.35) to [out=-70, in=180+70] (1.8,1.35);
\draw [dotted, red] (0.2 + 3, 1.35) to [out=270, in=180] (1.8 + 3,0);
\draw [dotted, red] (0.2 + 3, 0) to [out=0, in=270] (1.8 + 3,1.35);
\draw [dotted, red] (6.2, 1.35) to [out=-70, in=180+70] (6.2+1.8,1.35);

\draw [rounded corners=20] (0,0) rectangle (8, 3);

\foreach \i in {1,2}
{
  \draw (-1.5 + 3*\i,1.5) to [out=-30, in=210] (0.5 + 3*\i, 1.5);
  \draw (-1.2 + 3*\i,1.35) to [out=30, in=150] (0.2 + 3*\i, 1.35);
  \draw [red] (-0.5 + 3*\i, 1.5) circle (1.25 and 0.65);
}

\draw [red] (0, 1.35) to [out=70, in=180-70] (1.8,1.35);
\draw [red] (0.2 + 3, 1.35) to [out=20, in=45] (1.8 + 3,0);
\draw [red] (0.2 + 3, 0) to [out=135, in=160] (1.8 + 3,1.35);
\draw [red] (6.2, 1.35) to [out=70, in=180-70] (6.2+1.8,1.35);

\node at (0.9, 2.1) {$a$};
\node at (2.5, 2.4) {$b$};
\node at (5.2, 0.4) {$c$};
\node at (2.8, 0.4) {$d$};
\node at (5.5, 2.4) {$e$};
\node at (7.1, 2.1) {$f$};

\node [draw,shape=circle,fill=black,scale=0.4] at (4, 0.4) {};

\end{tikzpicture}}
\caption{Surfaces with exponential convergence.}
\end{figure}
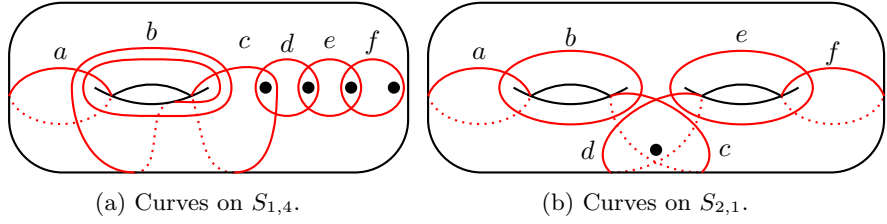

We now deal with surfaces with more punctures. If we have an exponential family on $S_{g,p}$ then by taking a common power we can obtain an additional fixed point. Removing this point gives an exponential family on $S_{g,p+1}$. Conversely, if we have an exponential family on $S_{g,p+1}$ and one of the punctures is a singularity with order at least two (for each mapping class in the family) then we may fill it and obtain an exponential family on $S_{g,p}$.

We now deal with surfaces of higher genus. Note that having an exponential family is preserved under taking covers. Thus if $g > 3$ is odd then $S_{g,0}$ is a cover of $S_{3,0}$ and so we may lift our exponential example from Section~\ref{sec:upper_bounds} to it. On the other hand, if $g > 3$ is even then $S_{g, 2}$ is a cover of $S_{2,2}$. Therefore, after first adding an additional puncture to our $S_{2,1}$ example by the preceding paragraph, we can lift this exponential example to $S_{g, 2}$. Now note that for this lifted family the punctures are both singularities of order at least $g / 2 \geq 2$ and so can be filled. Hence we can construct an exponential family on $S_{g,0}$ in this case too.

\subsection{Polynomial convergence}

When $3g - 3 + p$ is even lower, a polynomial spectral ratio bound still holds for:
\begin{itemize}
\item $S_{0,5}$ via the (spherical) braid $\sigma_3 \sigma_1^k \sigma_4^{-k} \sigma_2^{-1}$,
\item $S_{1,2}$ via $T_{c} \circ T_{a}^k \circ T_{d}^{-k} \circ T_{b}^{-1}$, and
\item $S_{2,0}$ via $T_{c} \circ T_{a}^k \circ T_{d}^{-k} \circ T_{b}^{-1}$.
\end{itemize}
Again, the curves used for $S_{1,2}$ and $S_{2,0}$ are shown in Figure~\ref{fig:S_1_2} and Figure~\ref{fig:S_2_0} respectively.

\begin{figure}[ht]
\centering
\subcaptionbox{Curves on $S_{1,2}$. \label{fig:S_1_2}}{\begin{tikzpicture}[scale=0.75, thick]

\draw [dotted, red] (3, 1.45) to [out=120, in=180] (4.5,3);
\draw [dotted, red] (3, 1.275) to [out=240, in=180] (4.5,0);

\draw [rounded corners=20] (0,0) rectangle (6.75, 3);

\draw (-1.5 + 3,1.5) to [out=-30, in=210] (0.5 + 3, 1.5);
\draw (-1.2 + 3,1.35) to [out=30, in=150] (0.2 + 3, 1.35);

\draw [red] (-0.5 + 3, 1.5) circle (1.25 and 0.65);

\draw [red] (5.25,1.5) to [out=90,in=0] (2.5,2.75) to [out=180,in=90] (0.5,1.5) to [out=270,in=180] (2.5,0.25) to [out=0,in=270] (5.25,1.5);

\draw [red] (4.125,1.5) to [out=90,in=90] (6.375,1.5) to [out=270,in=270] (4.125,1.5);

\draw [red] (4.5,3) to [out=0,in=90] (5.625,1.5) to [out=270,in=0] (4.5,0);
\draw [red] (3, 1.45) to [out=40, in=90] (4.875, 1.5) to [out=270,in=25] (3, 1.275);

\node at (0.25, 1.5) {$a$};
\node at (6.55, 1.5) {$b$};
\node at (5.5, 2.7) {$c$};
\node at (1, 1.5) {$d$};

\node [draw,shape=circle,fill=black,scale=0.4] at (4.5, 1.5) {};
\node [draw,shape=circle,fill=black,scale=0.4] at (6, 1.5) {};

\end{tikzpicture}}
\subcaptionbox{Curves on $S_{2,0}$. \label{fig:S_2_0}}{\begin{tikzpicture}[scale=0.75, thick]

\draw [dotted, red] (0, 1.35) to [out=-70, in=180+70] (1.8,1.35);
\draw [dotted, red] (0.2 + 3, 1.35) to [out=-50, in=180+50] (1.8 + 3,1.35);

\draw [rounded corners=20] (0,0) rectangle (8, 3);

\foreach \i in {1,2}
{
  \draw (-1.5 + 3*\i,1.5) to [out=-30, in=210] (0.5 + 3*\i, 1.5);
  \draw (-1.2 + 3*\i,1.35) to [out=30, in=150] (0.2 + 3*\i, 1.35);
  \draw [red] (-0.5 + 3*\i, 1.5) circle (1.25 and 0.65);
}

\draw [red] (0, 1.35) to [out=70, in=180-70] (1.8,1.35);
\draw [red] (0.2 + 3, 1.35) to [out=70, in=180-70] (1.8 + 3,1.35);

\node at (0.9, 2.1) {$a$};
\node at (2.5, 2.4) {$b$};
\node at (4, 2.1) {$c$};
\node at (5.5, 2.4) {$d$};

\end{tikzpicture}}
\caption{Surfaces with polynomial convergence.}
\end{figure}
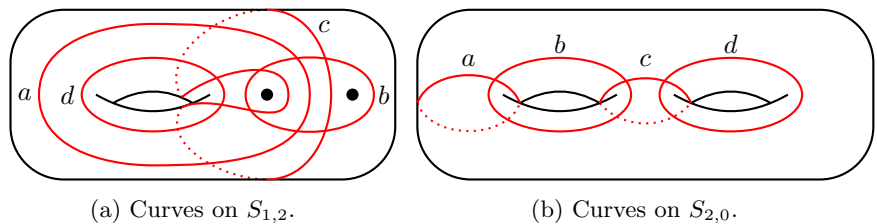

In the case of $S_{0,5}$, for example, if $k > 4$ and $4k + 1$ is not a square then the minimal polynomial of the dilatation of the pseudo-Anosov (spherical) braid $\sigma_3 \sigma_1^k \sigma_4^{-k} \sigma_2^{-1}$ is
\[ 1 - (2k + 5)x + (k^2 + 4k + 8)x^2 - (2k + 5)x^3 + x^4. \]
The same substitution trick allows us to explicitly compute the roots of this polynomial and so determine that the spectral ratio of this mapping classes is at most $1 + \frac{1}{\sqrt{k}}$.

Again, by taking a common power of these families we can obtain additional fixed points. By puncturing these out we can then also obtain a family of pseudo-Anosov mapping classes on $S_{0,6}$, $S_{1,3}$ and $S_{1,4}$ whose spectral ratios converge to one polynomially.

\subsection{No convergence}

For $S_{0,4}$, $S_{1,0}$ and $S_{1,1}$ the dilatation $\lambda$ of a pseudo-Anosov $h$ is a quadratic irrational. Thus $\lambda$ has a single Galois conjugate, its reciprocal, and so $\spectralratio(h) = \lambda^2$. However, since $\lambda + 1/\lambda \geq 3$ \cite[Section~5.1.3]{FarbMargalit} it follows that $\spectralratio(h) \geq \varphi^4 \approx 6.854101\cdots$.

Finally, there are no pseudo-Anosov mapping classes on $S_{0,0}$, $S_{0,1}$, $S_{0,2}$ or $S_{0,3}$.

\section{Lower bounds}
\label{sec:lower_bounds}

In this section we show that the behaviour seen in the previous examples, where the spectral ratio converges to one exponentially with the word length of $h$, is actually the worst possible.

\begin{proposition}
\label{prop:spectral_ratio_lower}
Suppose that $S$ is a surface and that $X$ is a finite generating set for $\Mod(S)$. If $h \in \Mod(S)$ is pseudo-Anosov then
\[ \spectralratio(h) \geq 1 + 2^{- O(|h|)} \]
where $|h|$ denotes the word length of $h$ with respect to $X$.
\end{proposition}

To prove this result, we first recall some facts about algebraic numbers.

\begin{definition}[{\cite[Section~3.4]{Waldschmidt}}]
The \emph{height} of a polynomial $f(x) = \sum a_i x^i \in \ZZ[x]$ is
\[ \height(f) \defeq \log(\max(|a_i|)). \]
The \emph{height} of an algebraic number $\alpha \in \QQbar$ is $\height(\alpha) \defeq \height(\mu_\alpha)$ where $\mu_\alpha \in \ZZ[x]$ is its minimal polynomial.
\end{definition}

\begin{fact}
\label{fct:heights}
If $\alpha, \beta \in \QQbar$ are algebraic numbers then:
\begin{itemize}
\item $\height(\alpha \pm \beta) \leq \height(\alpha) + \height(\beta) + 1$ \cite[Property~3.3]{Waldschmidt},
\item $\height(\alpha \beta) \leq \height(\alpha) + \height(\beta)$ \cite[Property~3.3]{Waldschmidt}, and
\item $\height(\alpha^{-1}) = \height(\alpha)$.
\end{itemize}
\end{fact}

Most importantly, algebraic numbers of bounded degree and height are bounded away from zero.

\begin{lemma}[{\cite[Lemma~10.3]{BasuPollackRoy}}]
\label{lem:algebraic_approximations}
If $\alpha \neq 0$ then
\[ -\log(|\alpha|) \leq \height(\alpha) + \deg(\alpha). \inlineQED \]
\end{lemma}

\begin{proof}[Proof of Proposition~\ref{prop:spectral_ratio_lower}]
Suppose that $h \in \Mod(S)$ is pseudo-Anosov. Let $\lambda = \lambda(h)$ be the dilatation of $h$ and $\lambda'$ a distinct Galois conjugate that maximises $|\lambda'|$. Hence $\spectralratio(h) = |\lambda / \lambda'|$.

Since $S$ and $X$ are fixed, we have that
\[ \height(\lambda), \; \height(\lambda') \in O(|h|) \inlineand \deg(\lambda), \; \deg(\lambda') \in O(1). \]
Now consider $\alpha \defeq |\lambda / \lambda'| - 1$. It follows from Fact~\ref{fct:heights} that
\[ \height(\alpha) \in O(|h|) \inlineand \deg(\alpha) \in O(1). \]

As $\lambda$ is a Perron number \cite[Page~405]{FarbMargalit}, the spectral ratio $\spectralratio(h) > 1$ and so $\alpha \neq 0$. Therefore, by Lemma~\ref{lem:algebraic_approximations} we have that
\[ -\log(|\alpha|) \in O(|h|). \]
Rearranging this we obtain that
\[ \spectralratio(h) \geq 1 + 2^{-O(|h|)} \]
as required.
\end{proof}

\section{Flipper}
\label{sec:flipper}

All examples in this paper were found and verified using the Python package \texttt{flipper}~\cite{flipper}. For example, the following Python script uses \texttt{flipper} to recreate the examples on $S_{0,7}$ that are given in Section~\ref{sub:exp}.

\begin{minipage}[c]{0.95\textwidth}
\begin{lstlisting}
import flipper

S = flipper.load('SB_7')

for k in range(7, 50):
	h = S.mapping_class('S_4' + 's_5S_6s_1S_2' * k + 's_3')
	f = h.dilatation().polynomial()
	X = sorted([abs(float(x)) for x in f.real_roots()])
	print(k, f, X[-1] / X[-2])
\end{lstlisting}
\end{minipage}

Random sampling of spectral ratios can also be done using \texttt{flipper}. Such experiments suggest that these exponentially slow examples are actually very rare. One such sampled distribution is shown in Figure~\ref{fig:distribution}. Curiously, the distribution of $\log(\lambda_1) / \log(\lambda_2)$ is essentially flat.

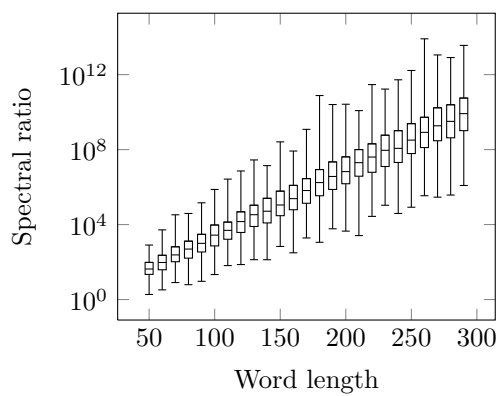
\begin{figure}[ht]
\centering
\begin{tikzpicture}[scale=1]
\begin{axis}[width=6.55cm, ymode=log, xlabel={Word length}, ylabel={Spectral ratio}, box plot width=1.5pt, ytick={1,10000,100000000,1000000000000}]
\boxplot [
  forget plot,
  box plot whisker bottom index=1,
  box plot box bottom index=2,
  box plot median index=3,
  box plot box top index=4,
  box plot whisker top index=5,
  ] {output_500.txt}
\end{axis}
\end{tikzpicture}
\caption{A sampled distribution (500 samples per word length) on $S_{0,7}$.}
\label{fig:distribution}
\end{figure}

\bibliographystyle{plain}
\bibliography{bibliography}

\end{document}